\documentclass[11pt]{article}
\usepackage{amsmath,amssymb,amsthm}
\usepackage{enumitem,cite}
\usepackage{authblk}
\usepackage{blindtext}
\usepackage[pagebackref=false,colorlinks,linkcolor=blue,citecolor=magenta]{hyperref}
\numberwithin{equation}{section}

\newcommand\mr[1]{\mathrm{#1}}
\newtheorem{definition}{Definition}[section]
\newtheorem{theorem}[definition]{Theorem}

\begin{document}
\title{A sequence of radially symmetric weak solutions for some nonlocal elliptic problem in $\mathbb{R}^N$}
\author[a]{M. Makvand Chaharlang\thanks{moloudmakvand@gmail.com}}
\author[b]{M. A. Ragusa\thanks{maragusa@dmi.unict.it}}
\author[a]{A. Razani\thanks{razani@sci.ikiu.ac.ir}}
\affil[a]{Department of Pure Mathematics, Faculty of Science, Imam Khomeini International University, Postal code: 34149-16818, Qazvin, Iran.}
\affil[b]{Department of Mathematics, University of Catania, Viale
Andrea Doria No. 6, Catania 95128, Italy, and RUDN University, Miklukho-Maklaya Str. 6, Moscow, 117198, Russia.}
\date{\empty}
\maketitle
\begin{abstract}
In this article a nonlocal elliptic problem involving $p$-Laplacian on unbounded domain is considered.
Using variational methods and under suitable conditions, the existence of a sequence of radially symmetric weak solutions (in two different cases) is proved.
\begin{description}
\item[Keywords:] Sequence of solutions, nonlocal elliptic problem, $p$-Laplacian, variational methods.
\end{description}
$2010$ Mathematics subject classification: 35J20, 34B10, 35J50.
\end{abstract}
\section{Introduction}
The theory of partial differential equations is a one of the powerful theories in mathematics.  On one side, this theory is used
 to model a wide variety
of physically significant problems arising in every different areas such as physics, engineering and other
applied disciplines (see \cite{Ehsani,hesaraki01,hesaraki02,Lindgren,Mokhtar,Mokhtarp,pournaki,pournakir,RaggusaDMJ, razani021, razani04,razani041,razani,razani19,razani02,razanigo,ScapllatoBVP}), and on the other side PDE's can be considered as an
instrument in the development of other branches of mathematics.

Many problems of mathematical physics and variational calculus are often very complicated. In all of them it is not
sufficient to deal with the classical solutions of differential equations. In fact, in the twentieth century, however, it was observed
that the space $C^k$ was not exactly the right space to
study solutions of differential equations. Although, solutions in $C^k$ are nicer but it seems people are happy with obtaining weak
solutions in some Sobolev space which are the modern replacement for these spaces that only satisfy the weak formulation. Therefore,
Sobolev spaces play an important role in the theory of partial differential equations. Since then, the study of the weak solution in other spaces such as
Orlicz-Morrey space and  $\dot{B}_{\infty,\infty}^{-1}$ space is a research problem (see \cite{GalaApplicable, GalaCMA, GalaAIMS}).

Laplace equation is the prototype for linear elliptic equations, as the most important partial differential equation of the second order.
This equation has a non-linear counterpart, the so-called $p$-Laplace
equation (see \cite{BehboudiFilomat, KhalkhaliInfinitely, Khalkhali2013, MahdaviFilomat,MakvandTJM,MakvandFilomat,MakvandGMJ}).
There has been a surge of interest in the $p$-Laplacian in many
different contexts from game theory to mechanics and image processing. Recently, a great attention has been focused on the study of
nonlocal operators of elliptic type (for example \cite{MakvandTJM,MakvandCKMS}), both, for research in pure Mathematics and for concrete
real world applications. From a physical point of view, the nonlocal operators play a crucial role in describing
several phenomena such as, the thin obstacle problem, optimization, material science,  geophysical fluid dynamics
and mathematical finance, phase transitions, water waves.

In this article we consider the following problem
\begin{equation}\label{eq1.1}
\begin{cases}
-M\left(\int_\mathbb{R^N} |\nabla u(x)|^p\,dx + \int_\mathbb{R^N}|u(x)|^pdx\right)(\Delta_p u+|u|^{p-2}u)=\lambda f(x,u) & \text{in $\mathbb{R}^N$},\\
 u\in W^{1,p}(\mathbb{R}^N)
\end{cases}
\end{equation}
where $p> N > 1$, $(\mathbb{R}^N,|.|)$ is the usual Euclidean space and $M:[0,+\infty[ \rightarrow \mathbb{R} $ is a continuous function and there exist two positive
constants $m_0$ and $m_1$ such that $m_0 \leq M(t) \leq m_1$ for all $t \geq 0$. We assume that
$f:\mathbb{R}^N \times \mathbb{R}\to \mathbb{R}$ is an $L^1$-Carath\'eodory function and $\lambda >0 $ is a real parameter.
Also $\Delta_pu:=\mr{div}(|\nabla u|^{p-2}\nabla u)$ denotes the $p$-Laplacian operator.
This equation is related to the stationary version of the Kirchhoff equation
\begin{equation}\label{eq23}
\rho \dfrac{\partial ^2 u}{\partial t^2} - (\dfrac{P_0}{h}+\dfrac{E}{2L}\int_0^L |
\dfrac{\partial u}{\partial x}|^2 \mr{d}x )\dfrac{\partial ^2 u}{\partial x^2}=0,
\end{equation}
proposed by Kirchhoff in 1883. This equation is an extension of the classical
d'Alembert's wave equation by considering the effects of the length changes of the string produced by
transverse vibrations. Since the first equation in \eqref{eq1.1} contains an integral over $\mathbb{R}^N $,
it is no longer a pointwise identity, and therefore it is often called a nonlocal problem. This problem models
several physical and biological systems, where $u$ describes a process which depends on the average of itself,
such as the population identity, (for example see \cite{chipot}). The parameters in \eqref{eq23} have the following meanings:
\emph{$h$} is the cross-section area, \emph{$E$} is the Young modulus, \emph{$\rho $} is the mass density, \emph{$L$}
is the length of the string, and \emph{$P_0$} is the initial tension. In recent years, $p$-Kirchhoff type problems have
been studied by many researchers, (for example see \cite{MakvandTJM,MakvandCKMS}). Beside elliptic problems with boundary conditions on
bounded domain of $\mathbb{R}^N$ which have extensive applications in different parts of scient and have been considered
by many authors recently, see \cite{MakvandGMJ,MakvandFilomat}, some elliptic problems arise on unbounded domain $\mathbb{R}^N$, see
\cite{Chen,Yang}. It is worth mentioning that one of the difficulties in studying problems on
unbounded domains is that there is no compact embedding for $W^{1,p}(\mathbb{R}^N)$. Although, we can use the continuous embedding
$W^{1,p}(\mathbb{R^N})\hookrightarrow L^{\infty}(\mathbb{R^N})$ from Morrey's theorem, it is far from being compact. Considering,
 a crucial embedding result due to Krist\'{a}ly and principally
based on a Strauss-type estimation which shows $W_r^{1,p}(\mathbb{R}^N)\hookrightarrow L^{\infty}(\mathbb{R}^N)$ is compact
whenever $2\leq N<p<+\infty$, where $W_r^{1,p}(\mathbb{R}^N)$ is a subspace of radially symmetric functions of $W^{1,p}(\mathbb{R}^N)$,
we prove the existence of a sequence of radially symmetric weak solutions for \eqref{eq1.1} which is a
nonlocal elliptic problem in the unbounded domain $\mathbb{R}^N$.

\section{Preliminaries}
In this part we remind some definitions and theorems as follows:
\begin{definition}\label{defn1}
The function $f:\mathbb{R}^N \times \mathbb{R}\to \mathbb{R}$ is said to be a\\
$L^1$-Carath\'eodory function, if
\begin{description}
\item[($A_1$)] the function $x\mapsto f(x,t)$ is measurable for every $t\in \mathbb{R}$,
\item[($A_2$)] the function $t\mapsto f(x,t)$ is continuous for a.e. $x\in \mathbb{R}^N$,
\item[($A_3$)] for every $h>0$ there exists a function $\ell_h\in L^1(\mathbb{R}^N)$ such that\\ $\sup_{|t|\leq h}|f(x,t)|\leq \ell_h(x)$,
for a.e. $x\in \mathbb{R}^N$.
\end{description}
\end{definition}

\begin{definition}
For fixed $\lambda$, a function $u:{\mathbb{R}^N}\rightarrow\mathbb{R}$ is said to be a weak solution of \eqref{eq1.1},
if $u \in W^{1,p}\left(\mathbb{R}^N\right)$ and for every $v\in W^{1,p}\left(\mathbb{R}^N\right)$
\[
\begin{array}{r}
M(\|u\|_{W^{1,p}\left(\mathbb{R}^N\right)})(\int_{\mathbb{R}^N}|\nabla
u(x)|^{p-2}\nabla u(x).\nabla v(x)dx +\int_{\mathbb{R}^N}|u(x)|^{p-2}u(x)v(x)dx)\\
-\lambda\int_{\mathbb{R}^N}f(x,u(x))v(x)dx=0,
\end{array}
\]
where $\|u\|_{W^{1,p}\left(\mathbb{R}^N\right)}:=\left(\int_{\mathbb{R}^N} |\nabla u(x)|^p dx + \int_\mathbb{R^N}|u(x)|^pdx\right)^{1/p}$.
\end{definition}

Later on, w define a functional energy where its the critical points are exactly the weak solutions of \eqref{eq1.1}.

Morrey's theorem, implies the continuous embedding
\begin{equation}\label{embedding}
W^{1,p}(\mathbb{R^N})\hookrightarrow L^{\infty}(\mathbb{R^N}),
\end{equation}
which says that there exists $c:=\frac{2p}{p-N}$, such that
$\|u\|_{\infty}\leq c\|u\|_{W^{1,p}(\mathbb{R}^N)}$, for every $u \in W^{1,p}(\mathbb{R^N})$, where $\|u\|_\infty:=\underset{x\in \mathbb{R}^N}{esssup}|u(x)|$, for every $u\in L^\infty(\mathbb{R}^N)$. Since in the low-dimensional case, every function $u\in W^{1,p}(\mathbb{R}^N)$
admits a continuous representation (see \cite[p.166]{Brezis}. In the sequel we will replace $u$ by this element.

We need the following notations (see \cite{candito} for more details):
\begin{itemize}
\item[(I)] $O(N)$ stands for the orthogonal group of $\mathbb{R}^N$.
\item[(II)] $B(0,s)$ denotes the open $N$-dimensional ball of center zero, radius $s>0$ and standard Lebesgue measure, $meas(B(0,s))$.
\end{itemize}

\begin{definition}\
\begin{itemize}
\item A function $h : \mathbb{R}^N\to \mathbb{R}$ is radially symmetric if $h(gx) = h(x)$, for every
$g\in O(N)$ and $x\in \mathbb{R}^N$.
\item Let $G$ be a topological group. A continuous map $\xi:G\times X\to X:(g,x)\to \xi(g,u):=gu$,
is called the action of $G$ on the Banach space $(X,\|.\|_X)$ if
\[
1u=u,\qquad (gm)u=g(mu),\qquad u\mapsto gu\ \text{is linear}.
\]
\item  The action is said to be isometric if $\|gu\|_X=\|u\|_X$, for every $g\in G$.
\item  The space of $G$-invariant points is defined by
\[
Fix(G):=\{ u\in X: gu=u, \text{for all}\ {g\in G}\}.
\]
\item A map $m:X\to \mathbb{R}$ is said to be $G$-invariant if $mog=m$ for every $g\in G$.
\end{itemize}
\end{definition}
The following theorem is important to study the critical point of the functional.
\begin{theorem}\label{palais}
(Palais (1979)) Assume that the action of the topological group $G$ on the Banach space $X$ is isometric.
If $J\in C^1(X:\mathbb{R})$ is $G$-invariant
and if $u$ is a critical point of $J$ restricted to $Fix(G)$, then $u$ is a critical point of $J$.
\end{theorem}

The action of the group $O(N)$ on $W^{1,p}(\mathbb{R}^N)$ can be defined by $(gu)(x):=u(g^{-1}x)$, for every
$g\in W^{1,p}(\mathbb{R}^N)$ and $x\in \mathbb{R^N}$.
A computation shows that this group acts linearly and isometrically, which means $\|u\|=\|gu\|$, for every
$g\in O(N)$ and $u\in W^{1,p}(\mathbb{R}^N)$.

\begin{definition}
The subspace of radially symmetric functions of $W^{1,p}(\mathbb{R^N})$ is defined by
\[
X:=W_r^{1,p}(\mathbb{R}^N):=\{ u\in W^{1,p}(\mathbb{R}^N): gu=u,\ \text{for all}\ g\in O(N)\},
\]
and endowed by the norm $\|u\|_r:=\left(\int_\mathbb{R^N} |\nabla u(x)|^p dx + \int_\mathbb{R^N}|u(x)|^pdx\right)^{1/p}$.
\end{definition}
The following crucial embedding result due to Krist\'{a}ly and principally based on a Strauss-type estimation (see \cite{Strauss})
(Also see \cite[Theorem 3.1]{Kristaly}, \cite{Varga} and \cite{Willem} for related subjects).
\begin{theorem}\label{thcom}
The embedding $W^{1,p}_r(\mathbb{R}^N)\hookrightarrow L^{\infty}(\mathbb{R}^N)$, is compact whenever $2\leq N<p<+\infty$.
\end{theorem}

Here we consider the following functionals:
\begin{itemize}
\item $F(x,\xi):=\int_{0}^{\xi}f(x,t)\mr{d}t,\quad $ for every
$(x,\xi)\in\mathbb{R^N}\times\mathbb{R}$,
\item $\widehat{M}(t):= \int_0^t M(s) \mr{d}s ,\quad$ for every $t>0$,
\item $\Phi(u):=\frac{1}{p}\widehat{M}(\|u\|_r^p) $ for every $u \in X$,
\item $\Psi(u):=\int_\mathbb{R^N}F(x,u(x))\mr{d}x$, for every $u \in X$,
\item $I_\lambda(u):=\Phi(u)-\lambda \Psi(u)$ for every $u \in X$.
\end{itemize}
By standard arguments \cite{candito}, we can show that $\Phi$ is G\^ateaux differentiable, coercive and
sequentially weakly lower semicontinuous whose derivative at the point $u \in X$ is the functional $\Phi'(u)\in X^*$ given by
\[
\Phi'(u)(v)=M(\|u\|^p_r)(\int_\mathbb{R^N}|\nabla u|^{p-2}\nabla u.\nabla v dx + \int_\mathbb{R^N}|u|^{p-2}uv dx),
\]
for every $v\in X$. Also standard arguments show that the functional $\Psi$ is well defined,
sequentially weakly upper semicontinuous and G\^ateaux differentiable whose G\^ateaux derivative
at the point $u \in X$ and for every $v\in X$ is given by,
\[
\Psi'(u)(v)=\int_\mathbb{R^N}f(x,u(x)) dx.
\]
In the last part of this section we recall the
following theorem \cite[Theorem 2.1]{bonanno}.
\begin{theorem}\label{theo2.1}
Let $X$ be a reflexive real Banach space, let $\Phi,\Psi:X\to \mathbb{R}$ be two G\^ateaux differentiable
functionals such that $\Phi$ is sequentially weakly lower semicontinuous, strongly continuous and coercive,
and $\Psi$ is sequentially weakly upper semicontinuous. For every $r>\inf_X\Phi$, put
\[
\begin{aligned}
\varphi(r)&:=\inf_{\Phi(u)<r}\frac{\sup_{\Phi(v)<r}\Psi(v)-\Psi(u)}{r-\Phi(u)},\\
\gamma&:=\liminf_{r\to+\infty}\varphi(r),\quad \text{and}\quad \delta:=\liminf_{r\to(\inf_X\Phi)^+}\varphi(r).
\end{aligned}
\]
Then the following properties hold:
\begin{itemize}
\item[($a$)] for every $r>\inf_X\Phi$ and every $\lambda \in]0,\frac{1}{\varphi(r)}[$, the restriction of the functional
\[ I_\lambda:=\Phi-\lambda \Psi \]
to $\Phi^{-1}(]-\infty,r[)$ admits a global minimum, which is a critical point (local minimum) of $I_\lambda$ in $X$.
\item[($b$)] if $\gamma<+\infty$, then for each $\lambda \in]0,\frac{1}{\gamma}[$, the following alternative holds either,
\begin{itemize}
\item[($b_1$)] $I_\lambda$ possesses a global minimum, or
\item[($b_2$)] there is a sequence $\{u_n\}$ of critical points (local minima) of $I_\lambda$ such that
\[ \lim_{n\to+\infty}\Phi(u_n)=+\infty, \]
\end{itemize}
\item[($c$)] if $\delta<+\infty$, then for each $\lambda \in]0,\frac{1}{\delta}[$, the following alternative holds either:
\begin{itemize}
\item[($c_1$)] there is a global minimum of $\Phi$ which is a local minimum of $I_\lambda$. or,
\item[($c_2$)] there is a sequence $\{u_n\}$ of pairwise distinct critical points (local minima) of
$I_\lambda$ which weakly converges to a global minimum of $\Phi$, with
\[ \lim_{n\to+\infty}\Phi(u_n)=\inf_{u\in X}\Phi(u). \]
\end{itemize}
\end{itemize}
\end{theorem}
\section{Multiple solutions}
In this part in order to present our main results we consider some assumptions as follows.
For fixed $D>0$, let
\[m(D):=meas(B(0,D))=D^N\frac{\pi^{\frac{N}{2}}}{\Gamma(1+\frac{N}{2})},\]
 where
$\Gamma$ is the Gamma function defined by $\Gamma(t):=\int_0^{+\infty}z^{t-1}e^{-z}dz$ for all $t>0$.  Moreover,
\begin{equation}\label{k}
 k:=\dfrac{1}{m_1 m(D){c^p}\left(\frac{\sigma(N,p)}{D^p}+g(p,N)\right)}\geq 0,
\end{equation}
where $\sigma(N,p):=2^{p-N}(2^N-1)$, $c= \frac{2p}{2-N}$, $m_1,m_0$ are upper and lower bounds for $M(t)$ in \eqref{eq1.1} and
\[g(p,N):=\frac{1+2^{N+p}N\int_{\frac{1}{2}}^1
t^{N-1}(1-t)^pdt}{2^N}.\]
Let $\|\cdot\|_1$ denotes the norm of $L^1(\Omega)$, we present our first result as the following theorem.
\begin{theorem}\label{theo1}
Let $f:\mathbb{R}^N\times\mathbb{R}\rightarrow\mathbb{R}$ be an
$L^1$- Carath\'{e}odory function, $f$ is radially symmetric with respect to first component, $F(x,t)\geq0$ for every
$(x,t)\in\mathbb{R^N}\times\mathbb{R^+}$ and $A<km_0 B$, where
$A:=\liminf_{\xi\rightarrow \infty}\frac{\|\ell_\xi\|_1}{\xi^{p-1}}$, $B:=\limsup_{\xi\rightarrow \infty}
\frac{\int_{B(0,\frac{D}{2})}F(x,\xi)dx}{\xi^p}$, $\ell_\xi\in L^1(\mathbb{R^N})$  and $k$ are given
by \eqref{defn1} and \eqref{k}, respectively. Then
for every
$$\lambda\in \Lambda:=\left]\frac{m_1 m(D)}{B}\left(\frac{\sigma(N,p)}{pD^p}+\frac{g(p,N)}{p}\right),\frac{m_0}{p c^p A}\right[,$$
there exists an unbounded sequence of radially symmetric weak solutions for \eqref{eq1.1} in $X$.
\end{theorem}
\begin{proof}
For fixed $\lambda \in \Lambda$, we consider $\Phi$, $\Psi$ and $I_\lambda$ as in
Section 2. Knowing that $\Phi$ and
$\Psi$ satisfy the regularity assumptions in Theorem \eqref{theo2.1}.  In order to study
the critical points of $I_\lambda$ in $X$, we show that $\lambda <\frac{1}{\gamma}<+\infty$, where $\gamma=\underset{r\to +\infty}{\liminf}\phi(r)$.

Let $\lbrace t_n \rbrace $ be a sequence of positive numbers such that $\lim_{n\to \infty} t_n=+\infty $ and
\[A=\lim_{n\rightarrow +\infty}\frac{\|\ell_{t_n}\|_1}{{t_n}^{p-1}}.\]
Set $r_n:=\dfrac{m_0{t_n}^p}{pc^p}$, for all $n \in \mr{N}$. Considering Theorem \ref{thcom} (by relation \eqref{embedding}), one has
\begin{equation}\label{phii}
\begin{aligned}
\Phi^{-1}(]-\infty,r_n[)&=\{v\in X;\, \Phi(v)<r_n\}\\
& = \{v\in X;\ \|v\|_r<(\frac{pr_n}{m_0})^{\frac{1}{p}}\}\\
& \subset \{v\in X;\ \|v\|_{\infty}<t_n\}.
\end{aligned}
\end{equation}
Since $\Phi(0)=\Psi(0)=0$, we have
\[
\begin{aligned}
\varphi(r_n)&=\inf_{\Phi(u)<r_n}\frac{(\sup_{\Phi(v)<r_n}\Psi(v))-\Psi(u)}{r_n-\Phi(u)}\\
& \leq \frac{\sup_{\|v\|_{\infty}<t_n}\int_{\mathbb{R}^N} F(x,v(x))dx}{\dfrac{m_0{t_n}^p}{pc^p}}\\
& \leq \frac{pc^p}{m_0}\frac{\|\ell_{t_n}\|_1}{{t_n}^{p-1}}.
\end{aligned}
\]
Hence, it follows that
\[
\gamma= \liminf_{n\to +\infty}\varphi(r_n)\leq \frac{pc^p}{m_0}\liminf_{n\to+\infty}\frac{\|\ell_{t_n}\|_1}{t_n^{p-1}}=\frac{pc^p}{m_0}A<+\infty.
\]
Also, from $\lambda \in \Lambda$, one has
$\dfrac{1}{\lambda} >\frac{pc^p}{m_0}A,$ and we get $\lambda <\dfrac{1}{\gamma}$.\\
Now, we show that $I_\lambda$ is unbounded from below. Let $\lbrace d_n\rbrace$ be a sequence of positive numbers such that $\lim_{n \rightarrow +\infty}d_n=+\infty$ and
\begin{equation}\label{B}
B=\lim_{n\rightarrow +\infty}
\frac{\int_{B(0,\frac{D}{2})}F(x,d_n)dx}{{d_n}^p}.
\end{equation}
Assume $\{v_n\} \subset X$ defined by
\begin{equation*}
 v_n(x):=\left\{\begin{array}{lll}
\displaystyle0  & \quad \mathbb{R^N}\setminus B(0,D)\\
\displaystyle d_n  &  \quad B(0,\frac{D}{2})
\\
\displaystyle \frac{2d_n}{D}(D-|x|)  & \quad B(0,D)\setminus
B(0,\frac{D}{2}),
\end{array}\right.
\end{equation*}
for every $n \in \mr{N}$. By a similar argument and compuations in \cite[P.1017]{candito} one can show that
$$\|v_n\|_r^p=d_n^pm(D)\left(\frac{\sigma(N,p)}{D^P}+g(p,N)\right).$$
Condition $(i)$, implies
$$\int_{\mathbb{R^N}}F(x,v_n(x))dx\geq\int_{B(0,\frac{D}{2})}F(x,d_n)dx,$$
for every $n \in N$. Therefore
\[\Phi(v_n)\leq \frac{m_1}{p}\|v_n \|_r^p=\frac{m_1 }{p}m(D) d_n^p\left(\frac{\sigma(N,p)}{D^p}+g(p,N)\right).\]
Let
\[\alpha:=m_1m(D)\left(\frac{\sigma(N,p)}{pD^P}+\dfrac{g(p,N)}{p}\right). \]
If $B <+\infty $,  for every $\varepsilon \in (\dfrac{\alpha}{\lambda B},1)$  and
$\varepsilon':=(1-\varepsilon)B >0$, \eqref{B} implies that there exists $N_\varepsilon$ such that for all $n>N_\varepsilon$,
\[
(\varepsilon-1)B <\frac{1}{d_n^p}\int_{B(0,\frac{D}{2})}F(x,d_n)dx - B<(1-\varepsilon) B.
\]
So
\[
\int_{B(0,\frac{D}{2})}F(x,d_n)dx>\varepsilon Bd_n^p,\qquad \text{for all}\ n>N_\varepsilon.
\]
Therefore for every $n> N_\varepsilon$,
\[
\begin{aligned}
I_\lambda(v_n)&\leq m_1d_n^pm(D)\left(\frac{\sigma(N,p)}{pD^p}+\dfrac{g(p,N)}{p}\right)-\lambda\varepsilon Bd_n^p\\
&=d_n^p(\alpha -\lambda \varepsilon B).
\end{aligned}
\]
Hence, $\lim_{n\rightarrow\infty}I_\lambda(v_n)=-\infty.$
\\If $B=+\infty$, let $L>\dfrac{\alpha}{\lambda}$. From \eqref{B} there exists $N_L$ such that for all $n>N_L$,
\[\int_{B(0,\frac{D}{2})}F(x,d_n)dx>Ld_n^p.\]
Thus
\[I_\lambda (v_n)\leq d_n^p\alpha-\lambda Ld_n^p=d_n^p(\alpha-\lambda L), \]
for every $n>N_L$, which shows that
\[\lim_{n\rightarrow\infty}I_\lambda(v_n)=-\infty.\]
Now, Theorem \ref{theo2.1} ($b$) implies, the functional $I_\lambda$ admits an unbounded sequence $\{u_n\}\subset X$ of
 critical points. Considering Theorem \ref{palais}, these critical points are also critical points for the smooth and $O(N)$-invariant functional $I_\lambda:W^{1,p}\left(\mathbb{R}^N\right)\to \mathbb{R}$. Therefore, there is a sequence of radially semmetric weak soluitions for the problem \eqref{eq1.1}, which are unbounded in $W^{1,p}\left(\mathbb{R}^N\right)$.
\end{proof}
Here we prove our second result which says that under different conditions the problem \eqref{eq1.1} has a sequence of
weak solutions, which converges weakly to zero.
\begin{theorem}\label{3.1}
Let $f:\mathbb{R}^N\times\mathbb{R}\rightarrow\mathbb{R}$ be an
$L^1$- Carath\'{e}odory function, $f$ is radially symmetric with respect to first component, $F(x,t)\geq0$ for every
$(x,t)\in\mathbb{R^N}\times\mathbb{R^+}$ and $A'<km_0 B'$
where
$ A':=\liminf_{\xi\rightarrow 0^+}\frac{\|\ell_\xi\|_1}{\xi^{p-1}}$, $B':=\limsup_{\xi\rightarrow 0^+}
\frac{\int_{B(0,\frac{D}{2})}F(x,\xi)dx}{\xi^p}$,
$\ell_\xi\in L^1(\mathbb{R^N})$  and $k$ are given by \eqref{defn1} and  \eqref{k}, respectively. Then
for every
$$\lambda\in \Lambda':=\left]\frac{m_1m(D)}{B'}\left(\frac{\sigma(N,p)}{pD^p}+\frac{g(p,N)}{p}\right),\frac{m_0}{p c^p A'}\right[,$$
there exists a sequence of weak solutions for \eqref{eq1.1} which converges weakly to zero in $X$.
\end{theorem}
\begin{proof}
For fixed $\lambda \in \Lambda'$, we consider $\Phi$, $\Psi$ and $I_\lambda$ as in Section 2. Knowing that
$\Phi$ and $\Psi$ satisfy the regularity assumptions in Theorem \eqref{theo2.1}, we show that $\lambda <\frac{1}{\delta}$. We know that $\inf_X\Phi =0$, therefore,
\[\delta:=\liminf_{r\to 0^+}\varphi(r).\]
Let $\lbrace t_n \rbrace $ be a sequence of positive numbers such that $\lim_{n\to \infty} t_n=0 $ and
\[A'=\lim_{n\rightarrow +\infty}\frac{\|\ell_{t_n}\|_1}{{t_n}^{p-1}}\]
Set $r_n:=\dfrac{m_0{t_n}^p}{pc^p}$, for all $n \in \mr{N}$. Considering \eqref{phii}, one has
$\|v\|_{\infty}<t_n$ for every $v\in X$ with $\Phi(v)<r_n$.
Since $\Phi(0)=\Psi(0)=0$, we have
\[
\begin{aligned}
\varphi(r_n)&=\inf_{\Phi(u)<r_n}\frac{(\sup_{\Phi(v)<r_n}\Psi(v))-\Psi(u)}{r_n-\Phi(u)}\\
&\leq \frac{\sup_{\|v\|_{\infty}<t_n}\int_{\mathbb{R}^N} F(x,v(x))dx}{\dfrac{m_0{t_n}^p}{pc^p}}\\
&\leq \frac{pc^p}{m_0}\frac{\|\ell_{t_n}\|_1}{{t_n}^{p-1}}.
\end{aligned}
.\]
Hence, it follows that
\[\delta \leq \liminf_{n\to +\infty}\varphi(r_n)\leq pc^p\liminf_{n\to+\infty}\frac{\|\ell_{t_n}\|_1}{t_n^{p-1}}=pc^pA'<+\infty.\]
Also, from $\lambda \in \Lambda'$, one has
$\dfrac{1}{\lambda} > pc^pA',$ and we get $\lambda <\dfrac{1}{\delta}$.\\
Now, we show that zero is not a local minimum of $I_\lambda. $ Let $\lbrace d_n\rbrace$ be a sequence of
positive numbers such that $\lim_{n \rightarrow +\infty}d_n=0$ and
\begin{equation}\label{BB}
B=\lim_{n\rightarrow +\infty}
\frac{\int_{B(0,\frac{D}{2})}F(x,d_n)dx}{{d_n}^p}.
\end{equation}
Also consider $\{v_n\}\subset X$ similar to the proof of Theorem \ref{theo1}, defined by $\{d_n\}$ above. By the
same argument, we obtain that $I_\lambda(v_n)<0$ for $n$ large enough. Thus zero is not a local minimum of $I_\lambda$ as
 $\lim_{n\to+\infty}I_\lambda(v_n)<I_\lambda(0)=0$. Therefore, there exists a sequence $\{u_n\}\subset X$ of
critical points of $I_\lambda$ which converges weakly to zero in $X$ as  $\lim_{n\to+\infty}\Phi(u_n)=0$.

Again, Considering Theorem \ref{palais}, these critical points are also critical points for the smooth and $O(N)$-invariant functional $I_\lambda:W^{1,p}\left(\mathbb{R}^N\right)\to \mathbb{R}$. Therefore, there is a sequence of radially symmetric weak solutions for the problem \eqref{eq1.1}, which converges weakly to zero in $W^{1,p}\left(\mathbb{R}^N\right)$.
\end{proof}

\section*{Acknowledgements}
We would like to thank the anonymous reviewer for the careful reading of our manuscript and the valuable comment and suggestion.\\
The second author is partially supported by
I.N.D.A.M - G.N.A.M.P.A. 2019 and the ``RUDN University Program 5-100''.\\
This work is done when the third author is visiting University of Manitoba on a sabbatical leave from Imam Khomeini International University.

\end{document}